\documentclass[reqno,oneside,11pt]{amsart}
\usepackage[a4paper,text ={14cm, 22.5cm}, centering, top=3cm]{geometry}
\usepackage[utf8]{inputenc}
\usepackage[T1]{fontenc}
\usepackage{stmaryrd,mathdots}
\usepackage[dvipsnames,svgnames,x11names,hyperref]{xcolor}
\usepackage[english]{babel}

\usepackage{lscape} 
\usepackage[shortlabels]{enumitem}
\setlist[enumerate]{label=\rm{(\arabic*)}}
\setlist[enumerate,2]{label=\rm({\it\roman*})}
\setlist[itemize]{label=\raisebox{0.25ex}{\tiny$\bullet$}}
\usepackage{fourier} 
\usepackage{ctable}
\usepackage{amssymb,amsmath}

\theoremstyle{plain}
\newtheorem{theorem}{Theorem}[section]
\newtheorem{corollary}[theorem]{Corollary}
\newtheorem{proposition}[theorem]{Proposition}
\newtheorem{lemma}[theorem]{Lemma}

\newtheorem{theoalph}{Theorem}

\newtheorem{coralph}[theoalph]{Corollary}
\newtheorem{proalph}[theoalph]{Proposition}

\theoremstyle{definition}

\newtheorem*{question}{Question}

\newtheorem{remark}[theorem]{Remark}

\addtocounter{section}{0}            

\newcommand{\smallO}[1]{\ensuremath{\mathop{}\mathopen{}{\scriptstyle\mathcal{O}}}}

\usepackage[pdfauthor={}, colorlinks, linktocpage, citecolor = Maroon, linkcolor = Maroon, urlcolor = Maroon]{hyperref}
\usepackage[all]{hypcap}        

\title{Cremona maps and involutions}

\author{Julie D\'eserti}

\date{\today}

\begin{document}

\maketitle

\begin{abstract}
We deal with the following question of Dolgachev : is the Cremona group generated
by involutions ? Answer is yes in dimension $2$ (\emph{see} 
\cite{CerveauDeserti}).
We give an upper bound of the minimal number $\mathfrak{n}_\varphi$ of involutions 
we need to write a birational self map~$\varphi$ of $\mathbb{P}^2_\mathbb{C}$. 

We prove that de Jonquières maps of $\mathbb{P}^3_\mathbb{C}$ and maps of small 
bidegree of $\mathbb{P}^3_\mathbb{C}$ can be written as a composition of involutions 
of $\mathbb{P}^3_\mathbb{C}$ and give an upper bound of $\mathfrak{n}_\varphi$ for such maps $\varphi$.
We get similar results in particular for automorphisms of 
$(\mathbb{P}^1_\mathbb{C})^n$, automorphisms of~$\mathbb{P}^n_\mathbb{C}$, 
tame automorphisms of $\mathbb{C}^n$, monomial maps of $\mathbb{P}^n_\mathbb{C}$, 
and elements of the subgroup generated by the standard involution of
$\mathbb{P}^n_\mathbb{C}$ and $\mathrm{PGL}(n+1,\mathbb{C})$.
\end{abstract}

\section{Introduction}

This article is motivated by the following question:

\begin{question}[Dolgachev]
Is the $n$ dimensional Cremona group generated by involutions ?
\end{question}

Answer is yes in dimension $2$; more precisely:

\begin{proposition}[\cite{CerveauDeserti}]\label{pro:dim2complexe}
For any $\varphi$ in $\mathrm{Bir}(\mathbb{P}^2_\mathbb{C})$ there exist $A_0$, $A_1$, 
$\ldots$, $A_k$ in $\mathrm{Aut}(\mathbb{P}^2_\mathbb{C})$ such that
\[
\varphi=\Big(A_0\circ\sigma_2\circ A_0^{-1}\Big)\circ\Big(A_1\circ\sigma_2\circ A_1^{-1}\Big)\circ\ldots\circ\Big(A_k\circ\sigma_2\circ A_k^{-1}\Big)
\]
where $\sigma_2$ denotes the standard involution of $\mathbb{P}^2_\mathbb{C}$
\[
\sigma_2 \colon (z_0:z_1:z_2)\dashrightarrow(z_1z_2:z_0z_2:z_0z_1).
\]
\end{proposition}

Let us note that since $\mathrm{Bir}(\mathbb{P}^2_\mathbb{R})$ is generated by
$\mathrm{PGL}(3,\mathbb{R})$ and some involutions (\cite{Zimmermann}), 
any element of $\mathrm{Bir}(\mathbb{P}^2_\mathbb{R})$ can be written as a 
composition of involutions.

If $\varphi$ is an element of $\mathrm{G}$, then $\mathfrak{n}(\varphi,\mathrm{H})$ 
is the minimal number of involutions of 
$\mathrm{H}\supset\mathrm{G}$ we need to write $\varphi$. In dimension $2$
we get the following result:

\begin{theoalph}
If $\varphi$ is an automorphism of $\mathbb{P}^1_\mathbb{C}\times\mathbb{P}^1_\mathbb{C}$, 
then 
$\mathfrak{n}\big(\varphi,\mathrm{Aut}(\mathbb{P}^1_\mathbb{C}\times\mathbb{P}^1_\mathbb{C})\big)\leq 4$.

If $\varphi$ is an automorphism of $\mathbb{P}^2_\mathbb{C}$, then 
$\mathfrak{n}\big(\varphi,\mathrm{Aut}(\mathbb{P}^2_\mathbb{C})\big)\leq 8$.

If $\varphi$ belongs to the Jonquières subgroup 
$\mathrm{J}_2\subset\mathrm{Bir}(\mathbb{P}^2_\mathbb{C})$, then 
$\mathfrak{n}\big(\varphi,\mathrm{J}_2\big)\leq 10$.

If $\varphi$ is a birational self map of $\mathbb{P}^2_\mathbb{C}$ of degree $d$, then 
$\mathfrak{n}\big(\varphi,\mathrm{Bir}(\mathbb{P}^2_\mathbb{C})\big)\leq 10d-2$.
\end{theoalph}

One can be more precise for the well-known subgroup $\mathrm{Aut}(\mathbb{C}^2)$ of polynomial 
automorphisms of $\mathbb{C}^2$ of $\mathrm{Bir}(\mathbb{P}^2_\mathbb{C})$:

\begin{theoalph}
Let $\varphi$ be an element of $\mathrm{Aut}(\mathbb{C}^2)$ of degree $d$. Then 
$\mathfrak{n}\big(\varphi,\mathrm{Bir}(\mathbb{P}^2_\mathbb{C})\big)\leq 44+\frac{9d}{4}$.

More precisely,
\begin{itemize}
\item if $\varphi$ is affine, then 
$\mathfrak{n}\big(\varphi,\mathrm{Aut}(\mathbb{P}^2_\mathbb{C})\big)\leq 8$;

\item if $\varphi$ is elementary, then $\mathfrak{n}\big(\varphi,\mathrm{J}_2\big)\leq 10$;

\item if $\varphi$ is generalized H\'enon map, then either it is of jacobian $1$ and 
$\mathfrak{n}\big(\varphi,\mathrm{Aut}(\mathbb{C}^2)\big)=2$ or 
$\mathfrak{n}\big(\varphi,\mathrm{Bir}(\mathbb{P}^2_\mathbb{C})\big)\leq 11$;

\item if $d$ is prime, then 
$\mathfrak{n}\big(\varphi,\mathrm{Bir}(\mathbb{P}^2_\mathbb{C})\big)\leq 26$.
\end{itemize}
\end{theoalph}

What happens in higher dimension ? A first result is the following:

\begin{proalph}
\begin{itemize}
\item If $\varphi$ is an automorphism of $(\mathbb{P}^1_\mathbb{C})^n$, then 
$\varphi$ can be written as a composition of involutions of 
$(\mathbb{P}^1_\mathbb{C})^n$, and 
$\mathfrak{n}\big(\varphi,\mathrm{Aut}\big(\mathbb{P}^1_\mathbb{C}\big)^n\big)\leq~2n$.

\item If $\varphi$ is an automorphism of $\mathbb{P}^n_\mathbb{C}$, then 
$\varphi$ can be written as a composition of involutions of $\mathbb{P}^n_\mathbb{C}$, 
and $\mathfrak{n}\big(\varphi,\mathrm{Aut}(\mathbb{P}^n_\mathbb{C})\big)\leq 2(n+1)$.
\end{itemize}
\end{proalph}

Since any element of 
\[
\mathrm{G}_n(\mathbb{C})=\langle\sigma_n=\Big(\prod_{\stackrel{i=0}{i\not=0}}^{n}z_i:\prod_{\stackrel{i=0}{i\not=1}}^{n}z_i:\ldots:\prod_{\stackrel{i=0}{i\not=n}}^{n}z_i\Big),\,\mathrm{Aut}(\mathbb{P}^n_\mathbb{C})\rangle
\] 
can be written as a composition of conjugate involutions (\cite{Deserti:reg}) one gets 
that: 

\begin{theoalph}
Let $n\geq 3$.
Any element of the normal subgroup generated by $\mathrm{G}_n(\mathbb{C})$ in 
$\mathrm{Bir}(\mathbb{P}^n_\mathbb{C})$ can be written as a composition of involutions
of $\mathbb{P}^n_\mathbb{C}$.
\end{theoalph}

Furthermore one can give an upper bound of 
$\mathfrak{n}\big(\varphi,\mathrm{Bir}(\mathbb{P}^n_\mathbb{C})\big)$ when $\varphi$
belongs to the subgroup of tame automorphisms of $\mathbb{C}^n$:

\begin{theoalph}
Let $n\geq 3$.
Let $\varphi$ be a tame automorphism of $\mathbb{C}^n$ of degree $d$.
Then $\varphi$ can be written as a composition of involutions of $\mathbb{P}^n_\mathbb{C}$.
Moreover,
\begin{itemize}
\item if $\varphi$ is affine, then 
$\mathfrak{n}\big(\varphi,\mathrm{Aut}(\mathbb{P}^n_\mathbb{C})\big)\leq 2n+4$;

\item if $\varphi$ is elementary, then 
$\mathfrak{n}\big(\varphi,\mathrm{Bir}(\mathbb{P}^n_\mathbb{C})\big)\leq 2n+10$;

\item otherwise 
$\mathfrak{n}\big(\varphi,\mathrm{Bir}(\mathbb{P}^n_\mathbb{C})\big)\leq \frac{d}{4}(2n+7)+10n+32$.
\end{itemize}
\end{theoalph}

Let us recall (\emph{see} \cite{PanSimis}) that the Jonquières subgroup 
$\mathrm{J}_{\smallO{}}(1,\mathbb{P}^3_\mathbb{C})$ of $\mathrm{Bir}(\mathbb{P}^3_\mathbb{C})$ 
is given in the affine chart $z_3=1$ by 
\[
\left\{ 
\varphi=\big(\varphi_0(z_0,z_1,z_2),\psi(z_1,z_2)\big)\,\big\vert\,\varphi_0\in\mathrm{PGL}(2,\mathbb{C}[z_1,z_2]),\, \psi\in\mathrm{Bir}(\mathbb{P}^2_\mathbb{C})
\right\}.
\]
Denote by $\mathrm{Mon}(n,\mathbb{C})$ the group of monomial maps of 
$\mathbb{P}^n_\mathbb{C}$, and finally set
\begin{small}
\[
\mathrm{J}_n=\mathrm{PGL}(2,\mathbb{C}(z_1,z_2,\ldots,z_{n-1}))\times\mathrm{PGL}(2,\mathbb{C}(z_2,z_3,\ldots,z_{n-1}))\times\ldots\times\mathrm{PGL}(2,\mathbb{C}(z_{n-1}))\times\mathrm{PGL}(2,\mathbb{C})
\subset\mathrm{Bir}(\mathbb{P}^n_\mathbb{C}).
\]
\end{small}

\begin{theoalph} 
Assume that $2\leq\ell\leq 4$, and $n\geq 3$.
\begin{itemize}
\item If $\varphi\in\mathrm{Bir}(\mathbb{P}^3_\mathbb{C})$ is of bidegree 
$(2,\ell)$, then $\varphi$ can be written as a composition of involutions of 
$\mathbb{P}^3_\mathbb{C}$, and 
$\mathfrak{n}\big(\varphi,\mathrm{Bir}(\mathbb{P}^3_\mathbb{C})\big)\leq 9+7\ell$.

\item Any element $\varphi$ of $\mathrm{J}_{\smallO{}}(1;\mathbb{P}^3_\mathbb{C})$ 
of degree $d$ can be written as a composition of involutions of 
$\mathbb{P}^3_\mathbb{C}$, and 
$\mathfrak{n}\big(\varphi,\mathrm{Bir}(\mathbb{P}^3_\mathbb{C})\big)\leq 10d+6$.

\item If $\varphi$ belongs to $\mathrm{Mon}(n,\mathbb{C})$, then 
$\varphi$ can be written as a composition of involutions of~$\mathbb{P}^n_\mathbb{C}$, 
and $\mathfrak{n}\big(\varphi,\mathrm{Mon}(n,\mathbb{C})\big)\leq 3n+9$.

\item Any element $\varphi$ of $\mathrm{J}_n$ can be written as a composition of 
involutions of $\mathbb{P}^n_\mathbb{C}$, and 
$\mathfrak{n}\big(\varphi,\mathrm{J}_n\big)\leq 4(2n-1)$.
\end{itemize}
\end{theoalph}

If $\mathrm{H}$ is a subgroup of $\mathrm{G}$ let us denote by 
$\mathrm{N}(\mathrm{H};\mathrm{G})$ the normal subgroup generated by $\mathrm{H}$ in 
$\mathrm{G}$.

\begin{coralph}
Any element of 
\begin{small}
\begin{eqnarray*}
& & \langle\mathrm{N}\big(\mathrm{PGL}(4,\mathbb{C});\mathrm{Bir}(\mathbb{P}^3_\mathbb{C})\big),\,\mathrm{N}\big(\mathrm{J}_{\smallO{}}(1;\mathbb{P}^3_\mathbb{C});\mathrm{Bir}(\mathbb{P}^3_\mathbb{C})\big),\,\mathrm{N}\big(\mathrm{Mon}(3,\mathbb{C});\mathrm{Bir}(\mathbb{P}^3_\mathbb{C})\big),\,\\
& & \quad \mathrm{N}\big(\mathrm{G}_3(\mathbb{C});\mathrm{Bir}(\mathbb{P}^3_\mathbb{C})\big), \mathrm{N}\big(\langle\varphi_1,\ldots,\varphi_k\rangle;\mathrm{Bir}(\mathbb{P}^3_\mathbb{C})\big)\,\vert\,\varphi_i\in\mathrm{Bir}(\mathbb{P}^3_\mathbb{C})\text{ of bidegree $(2,\ell)$, $2\leq \ell\leq 4$}\rangle
\end{eqnarray*}
\end{small}
can be written as a composition of involutions of $\mathbb{P}^3_\mathbb{C}$.

\smallskip

For any $n\geq 4$, any element of 
\begin{small}
\[
\langle\mathrm{N}\big(\mathrm{PGL}(n+1,\mathbb{C});\mathrm{Bir}(\mathbb{P}^n_\mathbb{C})\big),\,\mathrm{N}\big(\mathrm{J}_n;\mathrm{Bir}(\mathbb{P}^n_\mathbb{C})\big),\,\mathrm{N}\big(\mathrm{Mon}(n,\mathbb{C});\mathrm{Bir}(\mathbb{P}^n_\mathbb{C})\big),\,\mathrm{N}\big(\mathrm{G}_n(\mathbb{C});\mathrm{Bir}(\mathbb{P}^n_\mathbb{C})\big)\rangle
\]
\end{small}
can be written as a composition of involutions of $\mathbb{P}^n_\mathbb{C}$.
\end{coralph}

\begin{remark}
An other motivation for studying birational maps of $\mathbb{P}^n_\mathbb{C}$
that can be written as a composition of involutions is the following. The 
group of birational maps of $\mathbb{P}^n_\mathbb{C}$ that can be written
as a composition of involutions is a normal subgroup of 
$\mathrm{Bir}(\mathbb{P}^n_\mathbb{C})$. So if the answer to Dolgachev Question
is no, we can give a negative answer to the following question asked by 
Mumford (\cite{Mumford}): is $\mathrm{Bir}(\mathbb{P}^n_\mathbb{C})$ 
a simple group~?
\end{remark}

\subsection*{Acknowledgments} I would like to thank D. Cerveau for his constant 
availability and kindness. Thanks to S. Zimmermann for pointing out Proposition
\ref{Pro:Zimmermann}.

\section{Recalls and definitions}

\subsection{Polynomial automorphisms of $\mathbb{C}^n$}\label{subsec:polyaut}

A \emph{polynomial automorphism} $\varphi$ of $\mathbb{C}^n$ is a bijective map 
from $\mathbb{C}^n$ into itself of the type
\[
(z_0,z_1,\ldots,z_{n-1})\mapsto \big(\varphi_0(z_0,z_1,\ldots,z_{n-1}),\varphi_1(z_0,z_1,\ldots,z_{n-1}),\ldots,\varphi_{n-1}(z_0,z_1,\ldots,z_{n-1})\big)
\]
with $\varphi_i\in\mathbb{C}[z_0,z_1,\ldots,z_{n-1}]$. The set of polynomial 
automorphisms of $\mathbb{C}^n$ form a group denoted 
$\mathrm{Aut}(\mathbb{C}^n)$.

Let $\mathrm{A}_n$ be the group of affine automorphisms of $\mathbb{C}^n$, 
and let $\mathrm{E}_n$ be the group of elementary automorphisms of $\mathbb{C}^n$. 
In other words $\mathrm{A}_n$ is the semi-direct product of 
$\mathrm{GL}(n,\mathbb{C})$ with the commutative unipotent subgroup of translations. 
Furthermore $\mathrm{E}_n$ is formed with automorphisms 
$(\varphi_0,\varphi_1,\ldots,\varphi_{n-1})$ of $\mathbb{C}^n$ where 
\[
\varphi_i=\varphi_i(z_i,z_{i+1},\ldots,z_{n-1})
\]
depends only on $z_i$, $z_{i+1}$, $\ldots$, $z_{n-1}$. The 
subgroup $\mathrm{Tame}_n$ of $\mathrm{Aut}(\mathbb{C}^n)$, called the group of 
tame automorphisms of $\mathbb{C}^n$, is the group generated by $\mathrm{A}_n$ 
and $\mathrm{E}_n$. For $n=2$ one has $\mathrm{Tame}_2=\mathrm{Aut}(\mathbb{C}^2)$, 
more precisely:

\begin{theorem}[\cite{Jung}]\label{thm:jung}
The group $\mathrm{Aut}(\mathbb{C}^2)$ has a structure of amalgamated product
\[
\mathrm{Aut}(\mathbb{C}^2)=\mathrm{A}_2\ast_{\mathrm{S}_2}\mathrm{E}_2
\]
with $\mathrm{S}_2=\mathrm{A}_2\cap\mathrm{E}_2$.
\end{theorem}

Nevertheless $\mathrm{Tame}_3\not=\mathrm{Aut}(\mathbb{C}^3)$
(\emph{see} \cite{ShestakovUmirbaev}).

\subsection{Birational maps of $\mathbb{P}^n_\mathbb{C}$}

A \emph{rational self map of $\mathbb{P}^n_\mathbb{C}$} is a map of the type
\[
(z_0:z_1:\ldots:z_n)\dashrightarrow \big(\varphi_0(z_0,z_1,\ldots,z_n):\varphi_1(z_0,z_1,\ldots,z_n):\ldots:\varphi_n(z_0,z_1,\ldots,z_n)\big)
\]
where the $\varphi_i$'s denote homogeneous polynomials of the same degree 
without common factor (of positive degree).

A \emph{birational self map $\varphi$ of $\mathbb{P}^n_\mathbb{C}$} is a rational 
map of $\mathbb{P}^n_\mathbb{C}$ such that there exists a rational self map~$\psi$ 
of $\mathbb{P}^n_\mathbb{C}$ with the following property 
$\varphi\circ\psi=\psi\circ\varphi=\mathrm{id}$ where 
$\mathrm{id}\colon(z_0:z_1:\ldots:z_n)\dashrightarrow(z_0:z_1:\ldots:z_n)$.

The \emph{degree} of $\varphi\in\mathrm{Bir}(\mathbb{P}^n_\mathbb{C})$ is the 
degree of the $\varphi_i$'s. For $n=2$, one has $\deg \varphi=\deg\varphi^{-1}$; for
$n=3$ such an equality does not necessary hold, we thus speak about 
the \emph{bidegree} of~$\varphi$ which is $(\deg\varphi,\deg\varphi^{-1})$.

The group of birational self maps of $\mathbb{P}^n_\mathbb{C}$ is 
denoted $\mathrm{Bir}(\mathbb{P}^n_\mathbb{C})$ and called \emph{Cremona group}.

The groups $\mathrm{Aut}(\mathbb{P}^n_\mathbb{C})=\mathrm{PGL}(n+1,\mathbb{C})$ and 
$\mathrm{Aut}(\mathbb{C}^n)$ are subgroups of $\mathrm{Bir}(\mathbb{P}^n_\mathbb{C})$. 

Let us mention that contrary to $\mathrm{Aut}(\mathbb{C}^2)$ the 
Cremona group in dimension $2$ does not decompose as a non-trivial
amalgam (appendix of \cite{CantatLamy}).

\subsection{Birational involutions in dimension $2$}

Let us first describe some involutions:

\begin{itemize}
\item Consider an irreducible curve $\mathcal{C}$ of degree $\nu\geq 3$ 
with a unique singular point $p$; assume
furthermore that $p$ is an ordinary multiple point with multiplicity 
$\nu-2$. To $(\mathcal{C},p)$ we can associate a birational involution 
$\mathcal{I}_J$ which fixes pointwise $\mathcal{C}$ and preserves lines
through $p$ as follows. Let $m$ be a generic point of 
$\mathbb{P}^2_\mathbb{C}\smallsetminus\mathcal{C}$; let $r_m$, $q_m$ and 
$p$ be the intersections of the line $(mp)$ with $\mathcal{C}$; the 
point $\mathcal{I}_J(m)$ is defined by: the cross ratio of $m$,
$\mathcal{I}_J(m)$, $q_m$ and $r_m$ is equal to $-1$. The map 
$\mathcal{I}_J$ is a \emph{de Jonquières involution} of 
$\mathbb{P}^2_\mathbb{C}$. A birational involution is 
\emph{of de Jonquières type} if it is birationally conjugate 
to a de Jonquières involution of $\mathbb{P}^2_\mathbb{C}$.

\item Let $p_1$, $p_2$, $\ldots$, $p_8$ be eight points of 
$\mathbb{P}^2_\mathbb{C}$ in general position. Consider the set
of sextics $\mathcal{S}=\mathcal{S}(p_1,p_2,\ldots,p_8)$ with 
double points at $p_1$, $p_2$, $\ldots$, $p_8$. Take a 
point~$m$ in $\mathbb{P}^2_\mathbb{C}$. The pencil given by the 
elements of $\mathcal{S}$ having a double point at $m$ has a 
tenth base double point point $m'$. The involution which 
switches $m$ and $m'$ is a \emph{Bertini involution}. A 
birational involution is \emph{of Bertini type} if it is 
birationally conjugate to a Bertini involution.

\item Let $p_1$, $p_2$, $\ldots$, $p_7$ be seven points of 
$\mathbb{P}^2_\mathbb{C}$ in general position. Denote by $L$  
the linear system of cubics through the $p_i$'s. Consider a 
generic point $p$ in $\mathbb{P}^2_\mathbb{C}$ and define 
by~$L_p$ the pencil of elements of $L$ passing through $p$. 
The involution which switches $p$ and the ninth base-point 
of $L_p$ is a \emph{Geiser involution}. A birational 
involution is \emph{of Geiser type} if it is birationally
conjugate to a Geiser involution.
\end{itemize}

Birational involutions of $\mathbb{P}^2_\mathbb{C}$ have been classified:

\begin{theorem}[\cite{Bertini}]
A non-trivial birational involution of $\mathbb{P}^2_\mathbb{C}$ is 
either of de Jonquières type, or of Bertini type, or of Geiser type.
\end{theorem}

\subsection{Birational involutions in higher dimension}

There are no classification in higher dimension; in \cite{Prokhorov:involution}
the author gives a first nice step toward a classification in dimension~$3$.

Let us give some examples:
\begin{itemize}
\item the involution 
\[
\sigma_n=\Big(\prod_{\stackrel{i=0}{i\not=0}}^{n}z_i:\prod_{\stackrel{i=0}{i\not=1}}^{n}z_i:\ldots:\prod_{\stackrel{i=0}{i\not=n}}^{n}z_i\Big)
\]

\item the involutions of $\mathrm{PGL}(n+1,\mathbb{C})$;

\item the involutions of $\mathrm{Mon}(n,\mathbb{C})$ induced by the 
involutions of $\mathrm{GL}(n,\mathbb{Z})$;

\item the de Jonquières involutions: consider a reduced hypersurface $H$
of degree $\nu$ in~$\mathbb{P}^n_\mathbb{C}$ that contains a linear 
subspace of points of multiplicity $\nu-2$. Assume that~$p$ is a 
singular point of $H$ of multiplicity $\nu-2$. Take a general point 
$m$ of~$H$. Denote by $\ell_p$ the line passing through $p$ and $m$.
The intersection of $\ell_p$ with $H$ contains~$p$ with multiplicity
$\nu-2$, and the residual intersection is a set of two points~$r_m$
and~$q_m$ in $\ell_p$. Define $\mathcal{I}_J(m)$ to be the point 
on $\ell_p$ such that the cross ratio of~$m$, $\mathcal{I}_J(m)$,
$q_m$ and $r_m$ are equal to~$-1$. The map $\mathcal{I}_J$ is a 
de Jonquières involution of $\mathbb{P}^n_\mathbb{C}$.
\end{itemize}

\section{Automorphisms of $(\mathbb{P}^1_\mathbb{C})^n$ and of $\mathbb{P}^n_\mathbb{C}$}

\begin{lemma}\label{lem:homography}
Any non-trivial homography is either an involution, or the composition of
 two involutions of $\mathrm{PGL}(2,\mathbb{C})$.

In particular if $\varphi$ belongs to 
$\mathrm{Aut}\big(\underbrace{\mathbb{P}^1_\mathbb{C}\times\mathbb{P}^1_\mathbb{C}\times\ldots\times\mathbb{P}^1_\mathbb{C}}_{\text{$n$ times}}\big)$, 
then 
\[
\mathfrak{n}\big(\varphi,\mathrm{Aut}\big(\underbrace{\mathbb{P}^1_\mathbb{C}\times\mathbb{P}^1_\mathbb{C}\times\ldots\times\mathbb{P}^1_\mathbb{C}}_{\text{$n$ times}}\big)\big)\leq 2n.
\]
\end{lemma}

\begin{remark}\label{rem:involution}
The homography $\nu\in\mathrm{PGL}(2,\mathbb{C})$ is a non-trivial 
involution if and only if there exists 
$p\in\mathbb{P}^1_\mathbb{C}\smallsetminus\mathrm{Fix}(\nu)$ such that 
$\nu^2(p)=p$, where $\mathrm{Fix}(\nu)$ denotes the set of fixed points 
of~$\nu$.

Indeed assume that there exists 
$p\in\mathbb{P}^1_\mathbb{C}\smallsetminus\mathrm{Fix}(\nu)$ such that 
$\nu^2(p)=p$, then $\mathrm{Fix}(\nu)\not=\mathbb{P}^1_\mathbb{C}$ and 
so $\nu\not=\mathrm{id}$. If $m\in\{p,\,\nu(p)\}$, then $\nu^2(m)=m$. 
If $p\not\in\{p,\,\nu(p)\}$, the cross ratio of $p$, $\nu(p)$, 
$\nu(m)$, $m$ is equal to the cross ratio of $p$, $\nu(p)$, $\nu(m)$ 
and $\nu^2(m)$. This implies that $\nu^2(m)=m$. 
\end{remark}

\begin{lemma}\label{lem:involution}
Let $\nu\in\mathrm{PGL}(2,\mathbb{C})$ be an homography. Consider 
three points $a$, $b$, $c$ of $\mathbb{P}^1_\mathbb{C}$ such that $a$, 
$b$, $c$ are distinct, $a\not\in\mathrm{Fix}(\nu)$, 
$b\not\in\mathrm{Fix}(\nu)$ and $b\not=\nu(a)$. 

There exist two involutions $\iota_1$, $\iota_2\in\mathrm{PGL}(2,\mathbb{C})$ 
such that $\nu=\iota_2\circ \iota_1$.
\end{lemma}

\begin{proof}
Let us first prove that there exists two unique homographies $\iota_1$, 
$\iota_2\in\mathrm{PGL}(2,\mathbb{C})$ such that
\[
\left\{
\begin{array}{ll}
\iota_1(a)=\nu(b),  \, \,\, \iota_1(b)=\nu(a), \, \,\, \iota_1(\nu(a))=b;\\
\iota_2(\nu(a))=\nu(b), \, \,\, \iota_2(\nu(b))=\nu(a),  \, \,\, \iota_2(\iota_1(c))=\nu(c).
\end{array}
\right.
\]

Note that by assumptions $a$, $b$, $\nu(a)$ (resp. $\nu(b)$, $\nu(a)$, 
$b$) are pairwise distinct. Hence there exists a unique homography 
$\iota_1\in\mathrm{PGL}(2,\mathbb{C})$ that sends $a$, $b$, $\nu(a)$
 onto $\nu(b)$, $\nu(a)$, $b$.

The points $\nu(a)$, $\nu(b)$ and $\iota_1(c)$ are distinct. Assume 
by contradiction that $\iota_1(c)=\nu(a)$, then $\iota_1(c)=\iota_1(b)$. 
By injectivity of $\iota_1$, one has $c=b$: contradiction. Similarly 
$\iota_1(c)\not=\nu(b)=\iota_1(a)$ and $\nu(a)\not=\nu(b)$. Since $a$, $b$, 
$c$ are distinct, $\nu(a)$, $\nu(b)$ and $\nu(c)$ also. As a consequence 
there exists a unique homography $\iota_1\in\mathrm{PGL}(2,\mathbb{C})$ 
that sends $\nu(a)$, $\nu(b)$, $\iota_1(c)$ onto $\nu(b)$, $\nu(a)$, $\nu(c)$.

\bigskip

By assumption $b$ and $\nu(a)$ are distinct so $b$ does not belong to 
$\mathrm{Fix}(\iota_1)$. But $\iota_1^2(b)=\iota_1(\nu(a))=b$. According 
to Remark \ref{rem:involution} the homography $\iota_1$ is thus an 
involution. Similarly $\nu(a)$ and $\nu(b)$ are distinct but $\nu(a)$ 
and $\nu(b)$ are switched by $\iota_2$ hence $\iota_2$ is an involution 
(Remark \ref{rem:involution}). 

Since $\nu(p)=\iota_2\circ \iota_1(p)$ for $p\in\{a,\,b,\,c\}$ one gets 
$\nu=\iota_2\circ \iota_1$.
\end{proof}

\begin{proof}[Proof of Lemma \ref{lem:homography}]
Let $\nu$ be an homography. If $\nu=\mathrm{id}$, then 
$\nu=\iota\circ \iota$ for any involution~$\iota$. Assume now that 
$\nu\not=\mathrm{id}$; then $\nu$ has at most two fixed points. Let us 
choose $a$, $b$ in $\mathbb{P}^1_\mathbb{C}\smallsetminus\mathrm{Fix}(\nu)$. 
If $a\not=\nu(b)$ or if $b\not=\nu(a)$, then $\nu$ can be written as a 
composition of two involutions (Lemma \ref{lem:involution}). If $b=\nu(a)$ 
and $a=\nu(b)$, then $\nu^2(a)=a$ with $a\not\in\mathrm{Fix}(\nu)$; 
Remark \ref{rem:involution} thus implies that $\nu$ is an involution. 
\end{proof}

\begin{lemma}\label{lem:transvection}
Let $n\geq 2$ be an integer.
\begin{enumerate}
\item Let $\Bbbk$ be a commutative ring of any characteristic. If 
$\varphi$ is an element of $\mathrm{SL}(n,\Bbbk)$, then 
$\mathfrak{n}\big(\varphi,\mathrm{SL}(n,\Bbbk)\big)\leq 2(n+1)$.

\item Assume that $\Bbbk$ is an algebraically closed field, and that 
$\varphi$ belongs to $\mathrm{PGL}(n,\Bbbk)$. Then 
$\mathfrak{n}\big(\varphi,\mathrm{PGL}(n,\Bbbk)\big)\leq 2(n+1)$.

\item If $\varphi$ is an element of 
$\mathrm{PGL}(2,\mathbb{C}[z_1,z_2,\ldots,z_{n-1}])$, then 
$\mathfrak{n}\big(\varphi,\mathrm{Bir}(\mathbb{P}^n_\mathbb{C})\big)\leq 8$.
\end{enumerate}
\end{lemma}

\begin{proof}
\begin{enumerate}
\item Let us recall that an element of $\mathrm{SL}(n,\Bbbk)$ can 
be written as a composition of $\leq n+1$ transvections 
(\cite{Perrin}). But a transvection is a composition of two involutions 
so any an element of $\mathrm{SL}(n,\Bbbk)$ can be written as a 
composition of $\leq 2(n+1)$ involutions.

\item If $\Bbbk$ is algebraically closed, then 
$\mathrm{PSL}(n,\Bbbk)\simeq\mathrm{PGL}(n,\Bbbk)$ and one gets 
the result.

\item Let $g$ be an element of 
$\mathrm{PGL}(2,\mathbb{C}[z_1,z_2,\ldots,z_{n-1}])$; denote by 
$P(z_1,z_2,\ldots,z_{n-1})$ its determinant and by $h$ a scaling 
of scale factor $\frac{1}{P(z_1,z_2,\ldots,z_{n-1})}$. Then 
$h\circ g$ belongs to $\mathrm{SL}(2,\mathbb{C}(z_1,z_2,\ldots,z_{n-1}))$ 
and hence, according to the first assertion, can be written as 
a composition of $\leq 6$ involutions. But $h$ is as a composition 
of two involutions:
\[
\frac{1}{z_0P(z_1,z_2,\ldots,z_{n-1})}\circ\frac{1}{z_0}.
\]
As a result 
$\mathfrak{n}\big(\varphi,\mathrm{Bir}(\mathbb{P}^n_\mathbb{C})\big)\leq 8$.
\end{enumerate}
\end{proof}

\section{Dimension $2$}

\subsection{The real Cremona group}

There is an analogue to Proposition \ref{pro:dim2complexe} for the
real Cremona group. 

\begin{theorem}\label{thm:realCremonagroup}
Any element of $\mathrm{Bir}(\mathbb{P}^2_\mathbb{R})$ can be written 
as a composition of involutions of~$\mathbb{P}^2_\mathbb{R}$.
\end{theorem}

Theorem \ref{thm:realCremonagroup} directly follows from the simplicity 
of $\mathrm{PGL}(3,\mathbb{R})$ and the following statement:

\begin{proposition}[\cite{Zimmermann}]\label{Pro:Zimmermann}
The group $\mathrm{Bir}(\mathbb{P}^2_\mathbb{R})$ is generated by 
$\mathrm{PGL}(3,\mathbb{R})$, the set of standard quintic involutions
and the two following quadratic involutions
\begin{align*}
& (z_1z_2:z_0z_2:z_0z_1) && (z_0z_2:z_1z_2:z_0^2+z_1^2).
\end{align*}
\end{proposition}

\subsection{The de Jonquières subgroup}

An element of $\mathrm{Bir}(\mathbb{P}^2_\mathbb{C})$ is a 
\emph{de Jonquières map} if it preserves a rational fibration, 
{\it i.e.} if it is conjugate
to an element of 
\[
\mathrm{J}_2=\mathrm{PGL}(2,\mathbb{C}(z_1))\rtimes\mathrm{PGL}(2,\mathbb{C}).
\]
We will denote by $\widetilde{\mathrm{J}_2}$ the subgroup of 
birational maps that preserves fiberwise the fibration $z_1=$ 
constant, {\it i.e.} 
$\widetilde{\mathrm{J}_2}=\mathrm{PGL}(2,\mathbb{C}(z_1))$.

\begin{lemma}\label{lem:J0}
If $\varphi=(\varphi_0,\varphi_1)$ belongs to $\widetilde{\mathrm{J}_2}$, 
then $\mathfrak{n}\big(\varphi,\widetilde{\mathrm{J}_2}\big)\leq 8$.

Furthermore if $\det \varphi_0=\pm 1$, then 
$\mathfrak{n}\big(\varphi,\widetilde{\mathrm{J}_2}\big)\leq 4$.
\end{lemma}

\begin{proof}
Lemma \ref{lem:transvection} implies the first assertion, and the 
last assertion follows from \cite{GustafsonHalmosRadjavi}.
\end{proof}

\begin{corollary}\label{cor:jonquieres}
Any de Jonquières map of $\mathbb{P}^2_\mathbb{C}$ can be written 
as a composition of $\leq 10$ Cremona involutions of 
$\mathbb{P}^2_\mathbb{C}$.
\end{corollary}

\begin{proof}
Let us remark that any Jonquières map $\varphi$ of 
$\mathbb{P}^2_\mathbb{C}$ can be written as $\psi\circ j\circ\psi^{-1}$ 
where~$\psi$ denotes an element of $\mathrm{Bir}(\mathbb{P}^2_\mathbb{C})$ 
and $j$ an element of $\mathrm{J}_2$. But
\[
j=\left(\frac{a(z_1)z_0+b(z_1)}{c(z_1)z_0+d(z_1)},\frac{\alpha z_1+\beta}{\gamma z_1+\delta}\right)=\left(z_0,\frac{\alpha z_1+\beta}{\gamma z_1+\delta}\right)\circ\left(\frac{a(z_1)z_0+b(z_1)}{c(z_1)z_0+d(z_1)},z_1\right)
\]
As a consequence 
\[
\varphi=\left(\psi\circ\left(z_0,\frac{\alpha z_1+\beta}{\gamma z_1+\delta}\right)\circ\psi^{-1}\right)\circ\left(\psi\circ\left(\frac{a(z_1)z_0+b(z_1)}{c(z_1)z_0+d(z_1)},z_1\right)\circ\psi^{-1}\right)
\]

Then one concludes with Lemmas \ref{lem:homography} and \ref{lem:J0}.
\end{proof}

\subsection{Subgroup of polynomial automorphisms of $\mathbb{C}^2$}\label{sec:polynomialautomorphisms}

Note that there is no analogue to Proposition \ref{pro:dim2complexe} in 
the context of polynomial automorphisms of $\mathbb{C}^2$. For instance
the automorphism $(2z_0,3z_1)$ cannot be written as a composition of 
involutions in $\mathrm{Aut}(\mathbb{C}^2)$.

According to Lemma \ref{lem:transvection} and Corollary \ref{cor:jonquieres} one has 
the following result:

\begin{lemma}\label{lem:affineelementaire}
Let $\varphi$ be a polynomial automorphism of $\mathbb{C}^2$.

If $\varphi$ is an affine automorphism, then $\mathfrak{n}\big(\varphi,\mathrm{Aut}(\mathbb{P}^2_\mathbb{C})\big)\leq 8$.

If $\varphi$ is an elementary automorphism, then $\mathfrak{n}\big(\varphi,\mathrm{J}_2\big)\leq 10$.
\end{lemma}

An element $\varphi\in\mathrm{Aut}(\mathbb{C}^2)$ is a \emph{generalized Hénon
 map} if 
\[
\varphi=(z_1,P(z_1)-\delta z_0)
\]
where $\delta$ belongs to $\mathbb{C}^*$ and $P$ is an element of $\mathbb{C}[z_1]$ of degree $\geq 2$. Note that $\delta=\mathrm{jac}(\varphi)$. 

\begin{lemma}\label{lem:henonmaps}
Let $\varphi\in\mathrm{Aut}(\mathbb{C}^2$) be a generalized Hénon map. 
\begin{itemize}
\item If $\varphi$ has jacobian $1$, then $\mathfrak{n}\big(\varphi,\mathrm{Aut}(\mathbb{C}^2)\big)\leq 2$;

\item otherwise $\mathfrak{n}\big(\varphi,\mathrm{Bir}(\mathbb{P}^2_\mathbb{C})\big)\leq 11$.
\end{itemize}
\end{lemma}

\begin{proof}
Any generalized H\'enon map of jacobian $1$ can be written $\big(z_1,P(z_1)-z_0\big)$ 
and so is the composition of two involutions: 
$(z_1,P(z_1)-z_0)=(z_1,z_0)\circ\big(P(z_1)-z_0,z_1\big)$.

Let $\varphi$ be a generalized H\'enon map; then
\[
\varphi=(z_1,P(z_1)-\delta z_0)=(z_1,z_0)\circ\big(P(z_1)-\delta z_0,z_1\big).
\]
Note that $\big(P(z_1)-\delta z_0,z_1\big)$ is an elementary automorphism; 
therefore 
$\mathfrak{n}\big(\varphi,\mathrm{Bir}(\mathbb{P}^2_\mathbb{C})\big)\leq 1+10=11$ 
(Lemma \ref{lem:affineelementaire}).
\end{proof}

Friedland and Milnor proved that any polynomial automorphism of degree 
$d$ with~$d$ prime is conjugate via an affine automorphism either 
to a generalized H\'enon map or to an elementary automorphism 
(\cite[Corollary 2.7]{FriedlandMilnor}). Since any generalized 
H\'enon map is the composition of $(z_1,z_0)\in\mathrm{A}_2$ with an elementary 
map one gets that any polynomial automorphism of degree $d$ with 
$d$ prime can be written as $a_1ea_2$ with $a_i\in\mathrm{A}_2$ and 
$e\in\mathrm{E}_2$. Lemmas \ref{lem:affineelementaire} and 
\ref{lem:henonmaps} thus imply:

\begin{lemma}
If $\varphi\in\mathrm{Aut}(\mathbb{C}^2)$ is of degree $d$ with $d$ 
prime, then 
$\mathfrak{n}\big(\varphi,\mathrm{Bir}(\mathbb{P}^2_\mathbb{C})\big)\leq 26$.
\end{lemma}

A sequence $(\varphi_1,\varphi_2,\ldots,\varphi_k)$ of length 
$k\geq 1$ is a \textsl{reduced word}, representing the group element 
$\varphi=\varphi_k\circ\varphi_{k-1}\circ\ldots\circ\varphi_1$ if 
\begin{itemize}
\item each factor $\varphi_i$ belongs to either $\mathrm{A}_2$ or 
$\mathrm{E}_2$ but not to the intersection $\mathrm{A}_2\cap\mathrm{E}_2$, 

\item and no two consecutive factors belong to the same subgroup 
$\mathrm{A}_2$ or $\mathrm{E}_2$. 
\end{itemize}
It follows from Theorem \ref{thm:jung} 
that every element of $\mathrm{Aut}(\mathbb{C}^2)$ can be expressed as 
such a reduced word, unless it belongs to the intersection 
$\mathrm{S}_2=\mathrm{A}_2\cap\mathrm{E}_2$. The degree of any reduced 
word $\varphi=\varphi_k\circ\varphi_{k-1}\circ\ldots\circ\varphi_1$ is 
equal to the product of the degree of the factor $\varphi_i$ 
(\emph{see} \cite[Theorem 2.1]{FriedlandMilnor}). Hence take 
$\varphi\in\mathrm{Aut}(\mathbb{C}^2)$ of degree $d\geq 2$, then 
$\varphi$ is a reduced word 
$\varphi_k\circ\varphi_{k-1}\circ\ldots\circ\varphi_1$ and
\begin{itemize}
\item either there exists only one $\varphi_i$ of degree $>1$, then 
$\varphi=\varphi_3\circ\varphi_2\circ\varphi_1$ with $\deg\varphi_2>1$
and $\deg\varphi_1=\deg\varphi_3=1$; as a result
$\mathfrak{n}\big(\varphi,\mathrm{Bir}(\mathbb{P}^2_\mathbb{C})\big)\leq 26$
(Lemma \ref{lem:affineelementaire}),

\item or there exits at least two $\varphi_i$'s of degree $>1$, then 
$\mathfrak{n}\big(\varphi,\mathrm{Bir}(\mathbb{P}^2_\mathbb{C})\big)\leq \frac{9d}{4}+44$. 
Indeed let $(a_0,e_1,a_1,e_2,a_2,\ldots,e_k,a_k)$ be a reduced word 
representing $\varphi$. Any $e_i$ has degree $\geq 2$ and 
$\deg \varphi=\deg  e_1\deg e_2\prod_{i=3}^k\deg e_i$ hence 
$\prod_{i=3}^k\deg e_i\leq \frac{d}{4}$ and so $2(k-2)\leq \frac{d}{4}$. 
As a result $k\leq \frac{d}{8}+2$ and 
\begin{eqnarray*}
\mathfrak{n}\big(\varphi,\mathrm{Bir}(\mathbb{P}^2_\mathbb{C})\big)&\leq& (k+1)\mathfrak{n}\big(a_i,\mathrm{Bir}(\mathbb{P}^2_\mathbb{C})\big)+k\mathfrak{n}\big(e_i,\mathrm{Bir}(\mathbb{P}^2_\mathbb{C})\big)\\
&\leq& \left(\frac{d}{8}+3\right)\mathfrak{n}\big(a_i,\mathrm{Bir}(\mathbb{P}^2_\mathbb{C})\big)+\left(\frac{d}{8}+2\right)\mathfrak{n}\big(e_i,\mathrm{Bir}(\mathbb{P}^2_\mathbb{C})\big).
\end{eqnarray*}
\end{itemize}

One can thus state 

\begin{theorem}
Let $\varphi$ be a polynomial automorphism of $\mathbb{C}^2$ of 
degree $d$.
\begin{itemize}
\item If $\varphi$ is affine, 
$\mathfrak{n}\big(\varphi,\mathrm{Aut}(\mathbb{P}^2_\mathbb{C})\big)\leq 8$;

\item if $\varphi$ is elementary, then 
$\mathfrak{n}\big(\varphi,\mathrm{J}_2\big)\leq 10$;

\item if $\varphi$ is generalized H\'enon map, then either it 
is of jacobian $1$ and 
$\mathfrak{n}\big(\varphi,\mathrm{Aut}(\mathbb{C}^2)\big)=2$ 
or 
$\mathfrak{n}\big(\varphi,\mathrm{Bir}(\mathbb{P}^2_\mathbb{C})\big)\leq 11$;

\item if $d$ is prime, then 
$\mathfrak{n}\big(\varphi,\mathrm{Bir}(\mathbb{P}^2_\mathbb{C})\big)\leq 26$;

\item otherwise 
$\mathfrak{n}\big(\varphi,\mathrm{Bir}(\mathbb{P}^2_\mathbb{C})\big)\leq\frac{9d}{4}+44$.
\end{itemize}
\end{theorem}

\begin{corollary}
If $\varphi$ is a polynomial automorphism of $\mathbb{C}^2$ of 
degree $d$, then 
$\mathfrak{n}\big(\varphi,\mathrm{Bir}(\mathbb{P}^2_\mathbb{C})\big)\leq\frac{9d}{4}+44$.
\end{corollary}

\subsection{Birational maps}

\begin{theorem}\label{thm:main}
If $\varphi\in\mathrm{Bir}(\mathbb{P}^2_\mathbb{C})$ is of degree $d$, 
then 
$\mathfrak{n}\big(\varphi,\mathrm{Bir}(\mathbb{P}^2_\mathbb{C})\big)\leq 10d-2$.
\end{theorem}

Before proving Theorem \ref{thm:main} let us give a first and "bad" 
bound. Let $\varphi$ be a birational self map of 
$\mathbb{P}^2_\mathbb{C}$ of degree $d$. The number of base points of
 $\varphi$ is $\leq 2d-1$ and the map $\varphi$ can be written with 
$\leq 2(2d-1)$ blow ups. Since a blow up can be written as 
$A_1\circ\sigma_2\circ A_2\circ\sigma_2\circ A_3$ with 
$A_i\in\mathrm{PGL}(3,\mathbb{C})$ the map $\varphi$ can be written 
with $4(2d-1)$ involution $\sigma_2$ and $4(2d-1)+1$ elements of 
$\mathrm{PGL}(3,\mathbb{C})$. As a consequence $\varphi$ can be 
written as a composition of $\leq 4(2d-1)+8\big(4(2d-1)+1\big)=72d-28$ 
involutions.

\begin{proof}[Proof of Theorem \ref{thm:main}]
Let us recall that if $\varphi$ is a birational self map of 
$\mathbb{P}^2_\mathbb{C}$ of degree $d$, then there exists a de Jonquières
 map $\psi$ of $\mathbb{P}^2_\mathbb{C}$ such that $\deg (\varphi\circ\psi)<d$ 
(\emph{see} \cite{Castelnuovo}, \cite[Theorem 8.3.4]{AlberichCarraminana}).

As a result any $\varphi\in\mathrm{Bir}(\mathbb{P}^2_\mathbb{C})$ of 
degree $d\geq 1$ can be written as follows
\[
A\circ\big(\psi_1\circ j_1\circ\psi_1^{-1}\big)\circ\big(\psi_2\circ j_2\circ\psi_2^{-1}\big)\circ\ldots\circ\big(\psi_k\circ j_k\circ\psi_k^{-1}\big)
\]
with $A$ in $\mathrm{PGL}(3,\mathbb{C})$, 
$\psi_\ell\in\mathrm{Bir}(\mathbb{P}^2_\mathbb{C})$, $j_l$ in $\mathrm{J}_2$ 
and $k\leq d-1$. 

The statement follows from Lemma \ref{lem:transvection} and Corollary 
\ref{cor:jonquieres}.
\end{proof}

\section{Dimension $3$}

\subsection{de Jonquières maps in dimension $3$}

Let us recall that a de Jonquières map $\varphi$ of $\mathbb{P}^2_\mathbb{C}$ 
of degree $d$ is a plane Cremona map satisfying one of the following 
equivalent conditions:
\begin{itemize}
\item there exists a point $\smallO{}\in\mathbb{P}^2_\mathbb{C}$ such
 that the restriction of $\varphi$ to a general line passing through 
$\smallO{}$ maps it birationally to a line passing through $\smallO{}$;

\item $\varphi$ has homaloidal type $(d;d-1,1^{2d-2})$, {\it i.e.} 
$\varphi$ has $2d-1$ base points, one of multiplicity $d-1$ and 
$2d-2$ of multiplicity $1$;

\item up to projective coordinate changes (source and target) 
\[
\varphi=\big(z_0g_{d-1}+g_d:(z_0q_{d-2}+q_{d-1})z_1:(z_0q_{d-2}+q_{d-1})z_2\big)
\]
with $g_{d-1}$, $g_d$, $q_{d-2}$, $q_{d-1}\in\mathbb{C}[z_1,z_2]$ of 
degree $d-1$, resp. $d$, resp. $d-2$, resp. $d-1$. 
\end{itemize}

In \cite{PanSimis} Pan and Simis propose suitable generalizations of 
de Jonquières maps to higher dimensional space $\mathbb{P}^n_\mathbb{C}$,
 $n\geq 3$. More precisely they study elements of the Cremona group 
$\mathrm{Bir}(\mathbb{P}^n_\mathbb{C})$ satisfying a condition akin to 
the first alternative above: for a point 
$\smallO{}\in\mathbb{P}^n_\mathbb{C}$ and a positive integer $k$ they 
consider the Cremona transformations that map a general $k$-dimensional
 linear subspace passing through $\smallO{}$ onto another such subspace. 
Fixing the point $\smallO{}$ these maps form a subgroup 
$\mathrm{J}_{\smallO{}}(k;\mathbb{P}^n_\mathbb{C})$ of 
$\mathrm{Bir}(\mathbb{P}^n_\mathbb{C})$. For any $k\leq\ell$ the
 following inclusion holds (\cite{PanSimis})
\[
\mathrm{J}_{\smallO{}}(\ell;\mathbb{P}^n_\mathbb{C})\subset\mathrm{J}_{\smallO{}}(k;\mathbb{P}^n_\mathbb{C})
\]
Let us recall the following characterization of elements of 
$\mathrm{J}_{\smallO{}}(1;\mathbb{P}^n_\mathbb{C})$:

\begin{proposition}[\cite{Pan}]
Fix $\smallO{}=(0:0:\ldots:0:1)$. A Cremona map 
$\varphi\in\mathrm{Bir}(\mathbb{P}^n_\mathbb{C})$ belongs to 
$\mathrm{J}_{\smallO{}}(1;\mathbb{P}^n_\mathbb{C})$ if and only if 
\[
\varphi=\big(z_0g_{d-1}+g_d:(z_0q_{\ell-1}+q_\ell)t_1:(z_0q_{\ell-1}+q_\ell)t_2:\ldots:(z_0q_{\ell-1}+q_\ell)t_n\big)
\]
where 
\begin{itemize}
\item $g_d$, $g_{d-1}$, $q_\ell$, $q_{\ell-1}$, $t_1$, $\ldots$, $t_n\in\mathbb{C}[z_1,z_2,\ldots,z_n]$,

\item $\deg g_{d-1}=d-1$, $\deg g_d=d$, $\deg q_{\ell-1}=\ell-1$, $\deg q_\ell=\ell$,

\item $\deg t_i=d-\ell$ for $i\in\{1,\ldots,n\}$,

\item $(t_1:t_2:\ldots:t_n)\in\mathrm{Bir}(\mathbb{P}^{n-1}_\mathbb{C})$.
\end{itemize}
\end{proposition}

\begin{theorem}\label{thm:jonquieres}
Let $\varphi$ be an element of 
$\mathrm{J}_{\smallO{}}(1,\mathbb{P}^3_\mathbb{C})$ of degree $d$; then 
$\mathfrak{n}\big(\varphi,\mathrm{Bir}(\mathbb{P}^3_\mathbb{C})\big)\leq 10d+6$.
\end{theorem}

If $\mathrm{H}$ is a subgroup of $\mathrm{G}$ let us denote by 
$\mathrm{N}(\mathrm{H};\mathrm{G})$ the normal subgroup generated by 
$\mathrm{H}$ in $\mathrm{G}$.

\begin{corollary}
Any birational map of 
$\mathrm{N}\big(\mathrm{J}_{\smallO{}}(1,\mathbb{P}^3_\mathbb{C});\mathrm{Bir}(\mathbb{P}^3_\mathbb{C})\big)$ 
is a composition of involutions of $\mathrm{Bir}(\mathbb{P}^3_\mathbb{C})$.
\end{corollary}

\begin{proof}[Proof of Theorem \ref{thm:jonquieres}]
Any $\varphi$ in $\mathrm{J}_{\smallO{}}(1,\mathbb{P}^3_\mathbb{C})$ can 
be written in the affine chart $z_3=1$
\[
\varphi=\left(\frac{z_0A(z_1,z_2)+B(z_1,z_2)}{z_0C(z_1,z_2)+D(z_1,z_2)},\psi(z_1,z_2)\right)
\]
where 
\begin{align*}
&\frac{z_0A(z_1,z_2)+B(z_1,z_2)}{z_0C(z_1,z_2)+D(z_1,z_2)}\in\mathrm{PGL}(2,\mathbb{C}[z_1,z_2]),
&& \psi\in\mathrm{Bir}(\mathbb{P}^2_\mathbb{C}).
\end{align*}

Let us note that
\[
\varphi=\big(z_0,\psi(z_1,z_2)\big)\circ\left(\frac{z_0A(z_1,z_2)+B(z_1,z_2)}{z_0C(z_1,z_2)+D(z_1,z_2)},z_1,z_2\right).
\]
The map $\psi$ can be written as a composition of $\leq 10d-2$ involutions 
(Theorem \ref{thm:main}) and 
$\frac{z_0A(z_1,z_2)+B(z_1,z_2)}{z_0C(z_1,z_2)+D(z_1,z_2)}\in\mathrm{PGL}(2,\mathbb{C}[z_1,z_2])$ 
can be written as a composition of $\leq 8$ involutions (Lemma 
\ref{lem:transvection}). As a result  $\varphi$ is a composition of 
$10d+6$ or fewer involutions.
\end{proof}

\subsection{Maps of small bidegrees}

If $\varphi$ is a birational self map of $\mathbb{P}^3_\mathbb{C}$, then the 
bidegree of $\varphi$ is the pair $(\deg\varphi,\deg \varphi^{-1})$. Let us 
recall that $\deg \varphi^{-1}\leq\big(\deg\varphi\big)^2$. The left-right 
conjugacy is the following one
\begin{align*}
&\mathrm{PGL}(4,\mathbb{C})\times\mathrm{Bir}(\mathbb{P}^3_\mathbb{C})\times\mathrm{PGL}(4,\mathbb{C}) && (A,\varphi,B)\mapsto A\varphi B^{-1}.
\end{align*}

Pan, Ronga and Vust give birational self maps of $\mathbb{P}^3_\mathbb{C}$ 
of bidegree $(2,\cdot)$ up to left-right conjugacy, and show that there 
are only finitely many biclasses 
(\cite[Theorems 3.1.1, 3.2.1, 3.2.2, 3.3.1]{PanRongaVust}). In particular 
they show that the smooth and irreducible variety of birational self maps 
of $\mathbb{P}^3_\mathbb{C}$ of bidegree $(2,\cdot)$ has three irreducible 
components of dimension $26$, $28$, $29$. More precisely the component 
of dimension $26$ (resp. $28$, resp. $29$) corresponds to birational maps 
of bidegree $(2,4)$ (resp. $(2,3)$, resp. $(2,2)$). Let us denote by 
$\mathcal{O}(\varphi)$ the orbit of $\varphi$ under the left-right 
conjugacy.

\begin{proposition}
Let $\varphi$ be a birational self map of $\mathbb{P}^3_\mathbb{C}$ of
 bidegree $(2,2)$. Then $\varphi$ can be written as a composition of 
involutions of $\mathbb{P}^3_\mathbb{C}$. Furthermore 
$\mathfrak{n}\big(\varphi,\mathrm{Bir}(\mathbb{P}^3_\mathbb{C})\big)\leq 23$.
\end{proposition}

\begin{proof}
If $\varphi$ is a birational self map of $\mathbb{P}^3_\mathbb{C}$ of 
bidegree $(2,2)$, then up to left-right conjugacy $\varphi$ is one of 
the following (\cite{PanRongaVust})
\begin{small}
\begin{align*}
& f_1=\big(z_0z_3:z_1z_3:z_2z_3:z_0^2-z_1z_2\big)
&& f_2=\big(z_0z_3:z_1z_3:z_2z_3:z_1z_2\big)\\
& f_3=\big(z_0z_3:z_1z_3:z_2z_3:z_2^2\big)
&& f_4=\big(z_0z_3:z_2z_3:z_3^2:z_1z_3-z_0^2+z_1z_2\big)\\
& f_5=\big(z_0z_3:z_2z_3:z_3^2:z_1z_3-z_1z_2\big)
&& f_6=\big(z_0z_3:z_2z_3:z_3^2:z_1z_3-z_2^2\big)\\
& f_7=\big(z_0z_3-z_1z_2:z_1z_3:z_2z_3:z_3^2\big)
&& f_8=\big(z_0z_3:z_1z_3:z_2z_3:z_3^2\big)
\end{align*}
\end{small}

Note that $f_8=\mathrm{id}$, and that $f_1$, $f_2$, $f_3$ are involutions. 
Any element $\psi$ in $\mathcal{O}(f_i)$, $i\in\{1,\,2,\,3,\,8\}$, satisfies 
$\mathfrak{n}\big(\psi,\mathrm{Bir}(\mathbb{P}^3_\mathbb{C})\big)\leq~21$. 
The other $f_i$ are de Jonquières maps of $\mathbb{P}^3_\mathbb{C}$ so 
according to Theorem 
\ref{thm:jonquieres} can be written as compositions of involutions. 
Nevertheless to find a better bound for 
$\mathfrak{n}\big(\varphi,\mathrm{Bir}(\mathbb{P}^2_\mathbb{C})\big)$ we 
will give explicit decompositions.

First
\begin{small}
\begin{eqnarray*}
f_4 &=&\big(z_0:z_2+z_3:-z_2:z_1\big)\circ\big(z_0z_2:z_1z_2:z_3^2:z_2z_3\big)\circ\big(z_0z_2:z_1z_3:z_3^2:z_2z_3\big)\circ\\
& & \big(z_0z_3:z_0^2-z_1z_3:z_2z_3:z_3^2\big)\circ\big(z_0:z_1:-z_3:z_2+z_3\big)
\end{eqnarray*}
\end{small}
hence for any $\psi\in\mathcal{O}(f_4)$ one has 
$\mathfrak{n}\big(\psi,\mathrm{Bir}(\mathbb{P}^3_\mathbb{C})\big)\leq 23$ 
(which corresponds to two elements in $\mathrm{PGL}(4,\mathbb{C})$ and 
three involutions).

Second
\begin{small}
\[
f_5=(z_0:z_3-z_2:z_3:z_1)\circ(z_0z_2:z_1z_3:z_3^2:z_2z_3)\circ(z_0z_2:z_1z_2:z_3^2:z_2z_3)\circ(z_0:z_1:z_3-z_2:z_3).
\]
\end{small}
As a consequence 
$\mathfrak{n}\big(\psi,\mathrm{Bir}(\mathbb{P}^3_\mathbb{C})\big)\leq 22$
for any $\psi\in\mathcal{O}(f_5)$.

Third
\begin{small}
\[
f_6=\big(z_0:z_1:z_3:-z_2\big)\circ\big(z_0z_3:z_1z_3:z_1^2-z_2z_3:z_3^2\big)\circ\big(z_0:z_2:z_1:z_3\big).
\]
\end{small}
Therefore for any $\psi\in\mathcal{O}(f_6)$ one has the inequality 
$\mathfrak{n}\big(\psi,\mathrm{Bir}(\mathbb{P}^3_\mathbb{C})\big)\leq 21$.

Last
\begin{small}
\begin{eqnarray*}
f_7&=&\big(-z_0:z_1:z_2:z_3\big)\circ\big(-z_0z_3+z_1z_2:z_1z_3:z_2z_3:z_3^2\big).
\end{eqnarray*}
\end{small}
So $\mathfrak{n}\big(\psi,\mathrm{Bir}(\mathbb{P}^3_\mathbb{C})\big)\leq 21$
for any $\psi\in\mathcal{O}(f_7)$.
\end{proof}

\begin{proposition}
Any birational self map $\varphi$ of $\mathbb{P}^3_\mathbb{C}$ of 
bidegree $(2,3)$ can be written as a composition of involutions 
of $\mathbb{P}^3_\mathbb{C}$; moreover 
$\mathfrak{n}\big(\varphi,\mathrm{Bir}(\mathbb{P}^3_\mathbb{C})\big)\leq 30$.
\end{proposition}

\begin{proof}
If $\varphi$ is a birational self map of $\mathbb{P}^3_\mathbb{C}$ of 
bidegree $(2,3)$, then $\varphi\in\mathcal{O}(f_i)$ where $f_i$ is one of
the following map (\cite{PanRongaVust})
\begin{small}
\begin{align*}
& f_1=\big(-z_0z_1+z_0z_2:z_0z_3:-z_0z_1+z_1z_2:z_1z_3\big)
&& f_2=\big(z_0z_1:z_0z_2-z_1z_2:z_0z_3:z_1z_3\big)\\
& f_3=\big(z_0z_1-z_0z_2:z_0z_3:z_1z_2:z_1z_3\big)
&& f_4=\big(z_0z_1:z_0z_2:z_0z_3-z_1z_3:z_1^2\big)\\
& f_5=\big(z_0z_2:z_0z_3:z_1^2:z_1z_3\big)
&& f_6=\big(z_0z_1:z_0z_2-z_1^2:z_0z_3:z_1z_3\big)\\
& f_7=\big(z_0z_1:z_0z_2-z_1z_3:z_0z_3:z_1^2\big)
&& f_8=\big(z_0z_2-z_1^2:z_0z_3:z_1z_2:z_1z_3\big)\\
& f_9=\big(z_0^2:z_0z_1:z_1z_2:z_0z_3-z_1^2\big)
&& f_{10}=\big(z_0^2:z_0z_2:z_1z_2:z_0z_3-z_1^2\big)\\
& f_{11}=\big(z_0z_2+z_1^2:z_0^2:z_0z_1:z_0z_3-z_1z_2\big)
&&
\end{align*}
\end{small}

Let us give for any of these maps a decomposition with involutions 
and elements of $\mathrm{PGL}(4,\mathbb{C})$:
\begin{small}
\begin{eqnarray*}
f_1&=&\big(-z_2+z_0:-z_1+z_3:z_0:z_3\big)\circ\big(z_0z_1:z_3^2:z_2z_3:z_1z_3\big)\circ\big(z_0z_1:z_3^2:z_1z_2:z_1z_3\big)\circ\\
& & \big(z_0:z_3-z_1:z_2:z_3\big)\circ\big(z_0z_2:z_1z_2:z_3^2:z_2z_3\big)\circ\big(z_0z_2:z_1z_3:z_3^2:z_2z_3\big)\circ\\
& & \big(-z_0+z_2:z_1:z_2:z_3\big)\circ\big(z_0z_2:z_1z_3:z_3^2:z_2z_3\big)\circ\big(z_0z_2:z_1z_2:z_3^2:z_2z_3\big)\circ\\
& & \big(z_3^2:z_1z_3:z_0z_2:z_0z_3\big)\circ\big(z_1z_3^2:z_0z_3^2:z_0z_1z_2:z_0z_1z_3\big)\\
f_2&=&\big(z_1:z_2:z_0:z_3\big)\circ\big(z_0z_3:z_0z_1:z_2z_3:z_3^2\big)\circ\big(z_0+z_3:z_1:z_2:z_3\big)\circ\\
& & \big(z_3^2:z_0z_1:z_2z_3:z_0z_3\big)\circ\big(z_3^2:z_0z_1:z_0z_2:z_0z_3\big)\circ\big(z_0z_1:z_3^2:z_1z_2:z_1z_3\big)\circ\\
& & \big(z_0z_3:z_3^2:z_1z_2:z_1z_3\big)\circ\big(z_0-z_1:z_1:z_2:z_3\big)\\
f_3&=&\big(-z_0:z_1:z_2:z_3\big)\circ\big(-z_0z_3+z_1z_2:z_1z_3:z_2z_3:z_3^2\big)\circ\\
& & \big(z_3^2:z_1z_3:z_0z_2:z_0z_3\big)\circ\big(z_1z_3^2:z_0z_3^2:z_0z_1z_2:z_0z_1z_3\big)\\
f_4&=&\big(z_0+z_1:z_2:z_3:z_1\big)\circ\big(z_0z_1:z_1^2:z_1z_2:z_0z_3\big)\circ\\
& & \big(z_0-z_1:z_1:z_2:z_3\big)\circ\big(z_0z_1:z_1^2:z_0z_2:z_1z_3\big)\\
f_5&=&\big(z_2:z_0:z_1:z_3\big)\circ\big(z_1^2z_3:z_0z_1z_3:z_2z_3^2:z_0z_1^2\big)\circ\\
& &\big(z_1^2z_3:z_0z_1z_3:z_0z_2z_3:z_0z_1^2\big)\circ\big(z_0z_1:z_1z_3:z_2z_3:z_1^2\big)\circ\big(z_0z_3:z_1z_3:z_2z_3:z_1^2\big) \\
f_6&=&\big(z_1:-z_2:z_3:z_0\big)\circ\big(z_0z_3:z_1z_3:z_0z_1-z_2z_3:z_3^2\big)\circ\big(z_1z_3:z_0z_1:z_0z_2:z_0z_3\big)\\
f_7&=&\big(z_0:z_2:z_3:z_1\big)\circ\big(z_1^2:z_0z_1:z_0z_2:z_1z_3\big)\circ\big(z_1^2:z_0z_1:z_0z_2:z_0z_3\big)\circ\\
& &\big(z_0:z_1:z_2-z_3:z_3\big)\circ\big(z_1^2:z_0z_1:z_1z_2:z_0z_3\big)\circ\big(z_1^2:z_0z_1:z_0z_2:z_0z_3\big) \\
f_8&=&\big(z_1:z_0:z_2:z_3\big)\circ\big(z_0z_3:z_0z_2-z_1z_3:z_2z_3:z_3^2\big)\circ\\
& &\big(z_0z_1:z_3^2:z_1z_2:z_1z_3\big)\circ\big(z_0z_3:z_3^2:z_1z_2:z_1z_3\big)\\
f_9&=&\big(z_0:z_1:z_2:-z_3\big)\circ\big(z_0^2:z_0z_1:z_0z_2:-z_3z_0+z_1^2\big)\circ\\
& &\big(z_0z_1:z_0^2:z_0z_2:z_1z_3\big)\circ\big(z_0z_1:z_0^2:z_1z_2:z_1z_3\big)\\
f_{10}&=&\big(z_0:z_2:z_1:-z_3\big)\circ\big(z_0z_2^2:z_1z_2^2:z_2^3:-z_3z_2^2+z_0z_1^2\big)\circ\\
& & \big(z_0z_2:z_0z_1:z_0^2:z_2z_3\big)\circ\big(z_0z_2:z_1z_2:z_0^2:z_2z_3\big)\\
f_{11}&=&\big(-z_2:z_0:z_1:-z_3\big)\circ\big(z_0^2:z_0z_1:-z_0z_2-z_1^2:z_0z_3\big)\circ\big(z_0^2:z_0z_1:z_0z_2:-z_0z_3+z_1z_2\big)
\end{eqnarray*}
\end{small}
\end{proof}

\begin{proposition}
Let $\varphi$ be a birational self map of $\mathbb{P}^3_\mathbb{C}$ of 
bidegree $(2,4)$. Then $\varphi$ can be written as a composition of 
involutions of $\mathbb{P}^3_\mathbb{C}$. Furthermore 
$\mathfrak{n}\big(\varphi,\mathrm{Bir}(\mathbb{P}^3_\mathbb{C})\big)\leq 37$.
\end{proposition}

\begin{proof}
If $\varphi$ is a birational self map of $\mathbb{P}^3_\mathbb{C}$ of 
bidegree $(2,4)$, then $\varphi\in\mathcal{O}(f_i)$ where $f_i$ is one of 
the following map (\cite{PanRongaVust})
\begin{small}
\begin{align*}
& f_1=\big(z_1z_2:z_1z_3:z_2z_3:z_0z_3-z_1^2-z_2^2\big)
&& f_2=\big(z_1^2-z_1z_3:z_1z_2:z_2^2:z_0z_3\big)\\
& f_3=\big(z_1z_2:z_2z_3:z_1^2:z_0z_3-z_2^2\big)
&& f_4=\big(z_1^2:z_1z_3:z_2(z_1-z_3):z_0z_3-z_2^2\big)\\
& f_5=\big(z_1z_3:z_2z_3:z_1^2:z_0z_3-z_2^2\big)
&& f_6=\big(z_1z_2:z_2z_3:z_3^2:z_0z_3-z_1^2-z_2^2\big)\\
& f_7=\big(z_2^2:z_1z_2:z_1^2-z_2z_3:z_0z_3\big)
&& f_8=\big(z_1^2:z_1z_3:z_3^2-z_1z_2:z_0z_3-z_2^2\big)\\
& f_9=\big(z_3^2:z_1z_3:z_1^2-z_2z_3:z_0z_3-z_2^2\big)
&& f_{10}=\big(z_1^2-z_1z_3:z_2^2-z_2z_3:z_1z_2:z_0z_3\big)\\
& f_{11}=\big(z_1^2-z_1z_3:z_1z_2:z_2z_3:z_0z_3-z_2^2\big)
&&
\end{align*}
\end{small}

Let us give for any of these maps a decomposition with involutions 
and elements of $\mathrm{PGL}(4,\mathbb{C})$:
\begin{small}
\begin{eqnarray*}
f_1&=&\big(z_3:z_2:z_1:-z_0\big)\circ\big(-z_0z_1z_2+z_1^2z_3+z_2^2z_3:z_2z_3^2:z_1z_3^2:z_1z_2z_3\big)\circ\\
& & \big(z_0z_1z_2:z_2z_3^2:z_1z_3^2:z_1z_2z_3\big)\circ\big(z_0z_3^2:z_2z_3^2:z_1z_3^2:z_1z_2z_3\big)\\
& & \\
f_2&=&\big(z_3:z_1:z_2:z_0\big)\circ\big(z_0z_2:z_1z_2:z_1^2:z_2z_3\big)\circ\big(z_0:z_1:z_2:z_2-z_3\big)\circ\\
& &\big(z_0z_2:z_1z_2:z_1^2:z_2z_3\big)\circ\big(z_0z_1:z_2^2:z_1z_2:z_1z_3\big)\circ\big(z_0z_1:z_1^2:z_1z_2:z_2z_3\big)\circ\\
& & \big(z_0z_1:z_2^2:z_1z_2:z_1z_3\big)\circ\big(z_0z_2:z_1z_2:z_1^2:z_2z_3\big)\circ\big(z_0z_2z_3:z_1z_2z_3:z_1^2z_3:z_1^2z_2\big)\circ\\
& &\big(z_0z_3:z_3^2:z_1z_2:z_1z_3\big)\circ\big(z_0z_1:z_3^2:z_1z_2:z_1z_3\big)\circ\big(z_0z_2:z_1z_2:z_3^2:z_2z_3\big)\circ\\
& &\big(z_0z_3:z_1z_2:z_3^2:z_2z_3\big)\circ\big(z_0z_3:z_3^2:z_1z_2:z_1z_3\big)\circ\big(z_0z_1:z_3^2:z_1z_2:z_1z_3\big)\circ\\
& &\big(z_0z_2z_3:z_1z_2z_3:z_1^2z_3:z_1^2z_2\big)\circ\big(z_0z_2:z_1z_2:z_1^2:z_2z_3\big)\\
& & \\
f_3&=&\big(z_3:z_2:z_1:-z_0\big)\circ\big(z_0z_2:z_1^2:z_2z_3:z_1z_2\big)\circ\big(-z_0+z_2:z_1:z_2:z_3\big)\circ\\
& &\big(z_0z_2:z_2z_3:z_1z_3:z_1z_2\big)\circ\big(z_0z_3:z_2z_3:z_1z_3:z_1z_2\big)\\
& & \\
f_4&=&\big(z_1:z_1-z_3:z_2:z_0\big)\circ\big(z_0z_3:z_1z_3:z_1z_2:z_1^2\big)\circ\big(z_0z_3:z_1z_3:z_2z_3:z_1^2\big)\circ\\
& & \big(-z_0:z_1:z_2:z_1-z_3\big)\circ\big(-z_0z_1+z_2^2:z_1^2:z_1z_2:z_1z_3\big)\circ\\
& & \big(z_0z_1:z_1z_3:z_2z_3:z_1^2\big)\circ\big(z_0z_3:z_1z_3:z_2z_3:z_1^2\big)\\
& & \\
f_5&=&\big(z_3:z_2:z_1:z_0\big)\circ\big(z_0z_3:z_1z_3:z_1z_2:z_1^2\big)\circ\big(z_0z_3:z_1z_3:z_2z_3:z_1^2\big)\circ\big(-z_0:z_1:z_2:z_3\big)\circ\\
& & \big(-z_0z_1+z_2^2:z_1^2:z_1z_2:z_1z_3\big)\circ\big(z_0z_1:z_1z_3:z_2z_3:z_1^2\big)\circ\big(z_0z_3:z_1z_3:z_2z_3:z_1^2\big)\\
& & \\
f_6&=&\big(-z_1:z_2:z_3:-z_0-2z_1\big)\circ\big(z_0z_3:-z_1z_3+z_2^2:z_2z_3:z_3^2\big)\circ\big(z_0z_2:z_1z_3:z_3^2:z_2z_3\big)\circ\\
& & \big(z_0z_2:z_1z_2:z_3^2:z_2z_3\big)\circ\big(-z_0z_3+z_1^2:z_1z_3:z_2z_3:z_3^2\big)\circ\big(z_0:z_1+z_2:z_2:z_3\big)\\
& & \\
f_7&=&\big(z_2:z_1:z_3:z_0\big)\circ\big(z_0z_2:z_1z_2:z_2^2:z_1^2-z_2z_3\big)\circ\\
& & \big(z_0z_2z_3:z_1z_2z_3:z_1^2z_3:z_1^2z_2\big)\circ\big(z_0z_3^2:z_1z_2z_3:z_1^2z_3:z_1^2z_2\big)\\
& & \\
f_8&=&\big(z_1:z_3:z_2:z_0\big)\circ\big(z_0z_1:z_1^2:-z_1z_2+z_3^2:z_1z_3\big)\circ\big(z_0z_1^2:z_3^3:z_1z_2z_3:z_1z_3^2\big)\circ\\
& &\big(z_0z_1:z_3^2:z_1z_2:z_1z_3\big)\circ\big(-z_0:z_1:z_2:z_3\big)\circ\big(-z_0z_3+z_2^2:z_1z_3:z_2z_3:z_3^2\big)\\
& & \\
f_9&=&\big(z_3:z_1:z_2:-z_0\big)\circ\big(z_0z_3:z_1z_3:-z_2z_3+z_1^2:z_3^2\big)\circ\big(-z_0z_3+z_2^2:z_1z_3:z_2z_3:z_3^2\big)
\end{eqnarray*}
\end{small}

\begin{small}
\begin{eqnarray*}
f_{10}&=&\big(z_3-z_2:z_1-z_2:z_2:z_0\big)\circ\big(z_0z_3:z_1z_2:z_2z_3:z_2^2\big)\circ\big(z_0z_3:z_1z_3:z_2z_3:z_2^2\big)\circ\\
& & \big(z_0:z_1:z_2:z_2+z_3\big)\circ\big(z_0z_1:z_2^2:z_1z_2:z_1z_3\big)\circ\big(z_0:z_1:z_2:z_1-z_3\big)\circ\\
& & \big(z_0z_1z_3:z_2^2z_3:z_1z_2z_3:z_1z_2^2\big)\circ\big(z_0z_3^2:z_2^2z_3:z_1z_2z_3:z_1z_2^2\big)\\
& & \\
f_{11}&=&\big(z_2:z_1:z_1-z_3:-z_0\big)\circ\big(z_0z_2:z_1z_2:z_1^2:z_2z_3\big)\circ\big(z_0z_3:z_1z_3:z_1z_2:z_3^2\big)\circ\\
& & \big(-z_0+z_2:z_1:z_2:z_1-z_3\big)\circ\big(z_0z_3:z_1z_3:z_2z_3:z_1^2\big)\circ\\
& & \big(z_0z_2^2:z_1z_2z_3:z_1^2z_3:z_2z_3^2\big)\circ\big(z_0z_2z_3:z_1z_2z_3:z_1^2z_3:z_1^2z_2\big)
\end{eqnarray*}
\end{small}
\end{proof}

\section{Dimension $\geq 3$}

\subsection{The group generated by the automorphisms of $\mathbb{P}^n_\mathbb{C}$ and the Cremona involution}

Pan has proved that, as soon as $n\geq 3$, the subgroup generated by $\mathrm{Aut}(\mathbb{P}^n_\mathbb{C})$ and the involution $\sigma_n$
is a strict subgroup $\mathrm{G}_n(\mathbb{C})$ of $\mathrm{Bir}(\mathbb{P}^n_\mathbb{C})$. This subgroup has been studied in \cite{BlancHeden,Deserti:reg}, and in particular:

\begin{proposition}[\cite{Deserti:reg}]\label{pro:Gn}
For any $\varphi$ in $\mathrm{G}_n(\mathbb{C})$ there exist $A_0$, $A_1$, $\ldots$, $A_k$ in $\mathrm{Aut}(\mathbb{P}^n_\mathbb{C})$ such that
\[
\varphi=\Big(A_0\circ\sigma_n\circ A_0^{-1}\Big)\circ\Big(A_1\circ\sigma_n\circ A_1^{-1}\Big)\circ\ldots\circ\Big(A_k\circ\sigma_n\circ A_k^{-1}\Big)
\]
\end{proposition}

\begin{corollary}
Any element of $\mathrm{N}\big(\mathrm{G}_n(\mathbb{C});\mathrm{Bir}(\mathbb{P}^n_\mathbb{C})\big)$ can be written as a composition of involutions of $\mathbb{P}^n_\mathbb{C}$.
\end{corollary}

\subsection{The group of tame automorphisms}

As we already mentioned it, $\mathrm{Tame}_3$ does not coincide with $\mathrm{Aut}(\mathbb{C}^3)$ (\emph{see} \S\ref{subsec:polyaut}): the Nagata automorphism 
\[
N=\big(z_0+2z_1(z_0z_2-z_1^2)+z_2(z_0z_2-z_1^2)^2,z_1+z_2(z_0z_2-z_1^2),z_2\big)
\]
is not tame (\cite{ShestakovUmirbaev}). Note that since the Nagata automorphism is contained in $\mathrm{G}_3(\mathbb{C})$ (\emph{see} \cite{BlancHeden}), it can also be written as a composition of involutions (Proposition \ref{pro:Gn}). Since $\mathrm{G}_n(\mathbb{C})$ contains the group of tame polynomial automorphisms of $\mathbb{C}^n$ (\emph{see} \cite{Deserti:reg}) one gets that

\begin{proposition}
Any element of $\mathrm{N}(\mathrm{Tame}_n,\mathrm{Aut}(\mathbb{C}^n))$ is a composition of involutions of $\mathbb{P}^n_\mathbb{C}$. 
\end{proposition}

Can we give an upper bound for $\mathfrak{n}(\varphi,\mathrm{Bir}(\mathbb{P}^n_\mathbb{C})\big)$ when $\varphi\in\mathrm{Tame}_n$ ?

Set 
\begin{eqnarray*}
\mathrm{H}_1&=&\Big\{\Big(\alpha z_0+p(z_1),\sum_{i=1}^{n-1}a_{1,i}z_i+\gamma_1,\sum_{i=1}^{n-1}a_{2,i}z_i+\gamma_2,\ldots,\sum_{i=1}^{n-1}a_{n-1,i}z_i+\gamma_{n-1}\Big)\Big\vert\\
& &\hspace{1cm} p\in\mathbb{C}[z_1],\,\alpha,\,a_{i,j},\,\gamma_i\in\mathbb{C},\,\alpha\det(a_{ij})\not=0\Big\},
\end{eqnarray*}
\begin{eqnarray*}
\mathrm{H}_2&=&\Big\{\Big(\alpha z_0+\beta z_1+\gamma,\delta z_0+\sum_{i=1}^{n-1}a_{1,i}z_i+\gamma_1,\sum_{i=1}^{n-1}a_{2,i}z_i+\gamma_2,\ldots,\sum_{i=1}^{n-1}a_{n-1,i}z_i+\gamma_{n-1}\Big)\Big\vert\\
& &\hspace{1cm} \alpha,\,\beta,\gamma,\,\delta,\,a_{i,j},\,\gamma_i\in\mathbb{C},\,\det M(\alpha,\beta,\gamma,\delta,a_{i,j})\not=0\Big\}
\end{eqnarray*}
where
\[
M(\alpha,\beta,\gamma,\delta,a_{i,j})=
\left(
\begin{array}{c|ccccc}
\alpha & \beta & \gamma & 0 & \ldots & 0\\
\hline\\
\delta & & & & & \\
0 & & & a_{i,j} & & \\
\vdots & & & & & \\
0 & & & & & 
\end{array}
\right)
\]
One can check that 
\begin{eqnarray*}
\mathrm{H}_1\cap\mathrm{H}_2&=&\Big\{\Big(\alpha z_0+\beta z_1+\gamma,\sum_{i=1}^{n-1}a_{1,i}z_i+\gamma_1,\sum_{i=1}^{n-1}a_{2,i}z_i+\gamma_2,\ldots,\sum_{i=1}^{n-1}a_{n-1,i}z_i+\gamma_{n-1}\Big)\Big\vert\\ 
& & \hspace{1cm} \alpha,\,\beta,\gamma,\,\delta,\,a_{i,j},\,\gamma_i\in\mathbb{C},\,\det M(\alpha,\beta,\gamma,a_{i,j})\not=0\Big\}
\end{eqnarray*}
where 
\[
M(\alpha,\beta,\gamma,a_{i,j})=
\left(
\begin{array}{c|ccccc}
\alpha & \beta & \gamma & 0 & \ldots & 0 \\
\hline \\
0 & & & a_{i,j} & & \\
\vdots & & & & & \\
0 & & & & &  
\end{array}
\right)
\]

\begin{proposition}
Let $\varphi=\varphi_k\circ\varphi_{k-1}\circ\ldots\circ\varphi_1$ 
be a reduced word in the amalgamated product 
$\mathrm{H}_1\ast_{\mathrm{H}_1\cap\mathrm{H}_2}\mathrm{H}_2$.

The degree of $\varphi$ is equal to the product of the degree of 
the factors $\varphi_i$.
\end{proposition}

\begin{proof}
We follow the proof of \cite[Theorem 2.1]{FriedlandMilnor}.

Let $\psi=(\psi_0,\psi_1,\ldots,\psi_{n-1})$ be an element
of $\mathrm{H}_1\ast_{\mathrm{H}_1\cap\mathrm{H}_2}\mathrm{H}_2$
which satisfy the condition degrees 
$d_0=\deg\psi_1\geq\deg\psi_i$ for any $0\leq i\leq n-1$.

Now consider $\varphi=(\varphi_0,\varphi_1,\ldots,\varphi_{n-1})$
an element of $H_1\smallsetminus(H_1\cap H_2)$ of degree $d$; 
in particular $\varphi_0=\alpha z_0+p(z_1)$ with 
$d=\deg p\geq 2$. Denote by $\widetilde{\varphi_i}$ the 
components of $\varphi\circ\psi$. One has 
$dd_0=\deg\widetilde{\varphi_0}>\deg\widetilde{\varphi_i}$ 
for any $1\leq i\leq n-1$. 

Take $\phi$ in $H_2\smallsetminus(H_1\cap H_2)$, and set 
$\phi\circ\varphi\circ\psi=(\widetilde{\phi}_0,\widetilde{\phi_1},\ldots,\widetilde{\phi_{n-1}})$.
Then $dd_0=\deg\widetilde{\phi_1}\geq\deg\widetilde{\phi_i}$ for 
any $i\geq 0$.

As a result whenever we compose with an element of 
$H_1\smallsetminus(H_1\cap H_2)$ followed with an element of 
$H_2\smallsetminus(H_1\cap H_2)$ the degree will be multiply 
by $d$. The statement follows by induction.
\end{proof}

Let us now remark that $\langle\mathrm{H}_1,\,\mathrm{H}_2\rangle$ 
contains both $\mathrm{A}_n$ and $(z_0+z_1^2,z_1,z_2,\ldots,z_{n-1})$. 
Since 
$\mathrm{Tame}_n=\langle\mathrm{A}_n,\,(z_0+z_1^2,z_1,z_2,\ldots,z_{n-1})\rangle$ 
(\emph{see} \cite[Chapter 5.2]{vandenEssen}) any tame automorphism is a 
reduced word in $\mathrm{H}_1\ast_{\mathrm{H}_1\cap\mathrm{H}_2}\mathrm{H}_2$. 
Following what we did in \S\ref{sec:polynomialautomorphisms} one 
obtains:

\begin{theorem}
Let $\varphi$ be a tame automorphism of $\mathbb{C}^n$, $n\geq 3$, of degree $d$. 
\begin{itemize}
\item If $\varphi$ is affine, then 
$\mathfrak{n}\big(\varphi,\mathrm{Aut}(\mathbb{P}^n_\mathbb{C})\big)\leq 2n+4$;

\item if $\varphi$ is elementary, then 
$\mathfrak{n}\big(\varphi,\mathrm{Bir}(\mathbb{P}^n_\mathbb{C})\big)\leq 2n+10$;

\item otherwise 
$\mathfrak{n}\big(\varphi,\mathrm{Bir}(\mathbb{P}^n_\mathbb{C})\big)\leq \frac{d}{4}(2n+7)+10n+32$.
\end{itemize}
\end{theorem}

\begin{remark}
We cannot use this strategy to get a more precise statement for $G_n(\mathbb{C})$. 
Indeed using similar arguments as in the appendix of \cite{CantatLamy} one 
can prove that $G_n(\mathbb{C})$
has property $(FR)$; in particular, according to \cite{Serre} one has:

\begin{proposition}
The group $G_n(\mathbb{C})$ does not decompose as a non-trivial amalgam. 
\end{proposition}

More precisely if $G_n(\mathbb{C})$ is contained in an amalgam $G_1\ast_A~G_2$, then 
$G_n(\mathbb{C})$ is contained in a conjugate of either $G_1$ or $G_2$ 
(\emph{see} \cite{Serre}).
\end{remark}

\subsection{Monomial maps in any dimension}

Let $\mathbb{A}_\mathbb{C}^n$ be the affine space of dimension $n$. 
The multiplicative group $\mathbb{G}_m^n$ can be identified to the 
Zariski open subset $(\mathbb{A}^1_\mathbb{C}\smallsetminus\{0\})^n$ 
of $\mathbb{P}^n_\mathbb{C}$. Hence 
$\mathrm{Bir}(\mathbb{P}^n_\mathbb{C})$ contains the group of all 
algebraic automorphisms of the group $\mathbb{G}_m^n$, {\it i.e.} 
the group $\mathrm{Mon}(n,\mathbb{C})$ of \emph{monomial maps} 
$\mathrm{GL}(n,\mathbb{Z})$.

\begin{theorem}[\cite{Ishibashi}]
Let $n\geq 3$ be an integer. Any element $\varphi$ of 
$\mathrm{GL}(n,\mathbb{Z})$ can be written as a composition of
involutions of $\mathrm{GL}(n,\mathbb{Z})$, and
$\mathfrak{n}\big(\varphi,\mathrm{GL}(n,\mathbb{Z})\big)\leq 3n+9$.
\end{theorem}

\begin{corollary}
Let $\varphi$ be an element of $\mathrm{Mon}(n,\mathbb{C})$, 
with $n\geq 3$. Then $\varphi$ can be written as a composition
of involutions of $\mathrm{Mon}(n,\mathbb{C})$, and 
$\mathfrak{n}\big(\varphi,\mathrm{Mon}(n,\mathbb{C})\big)\leq 3n+9$.
\end{corollary}

\begin{remark}
If $n$ is even, then 
$\mathrm{Mon}(n,\mathbb{C})\subset\mathrm{G}_n(\mathbb{C})$ 
(\emph{see} \cite{BlancHeden}) ; Proposition \ref{pro:Gn} thus already 
says that any monomial map of $\mathrm{Bir}(\mathbb{P}^n_\mathbb{C})$ 
can be written as a composition of involutions but here we get 
two more informations: 
\begin{itemize}
\item a bound for the minimal number of involutions, 

\item and the fact that the involutions belong to 
$\mathrm{Mon}(n,\mathbb{C})$.
\end{itemize}
Furthermore Proposition \ref{pro:Gn} 
gives nothing for $\mathrm{Mon}(n,\mathbb{C})$ for $n$ odd since 
$\mathrm{Mon}(n,\mathbb{C})\not\subset\mathrm{G}_n(\mathbb{C})$ 
as soon as $n$ is odd (\cite{BlancHeden}).
\end{remark}

\begin{corollary}
Any element of 
$\mathrm{N}\big(\mathrm{Mon}(n,\mathbb{C});\mathrm{Bir}(\mathbb{P}^n_\mathbb{C})\big)$ 
is a composition of involutions of~$\mathbb{P}^n_\mathbb{C}$.
\end{corollary}

\subsection{Subgroups $\mathrm{J}_n$}

Let us introduce $\mathrm{J}_n$ the subgroup of 
$\mathrm{Bir}(\mathbb{P}^n_\mathbb{C})$ formed by the maps of 
the type
\[
\left(\varphi_0,\varphi_1,\ldots,\varphi_{n-2},\frac{\alpha z_{n-1}+\beta}{\gamma z_{n-1}+\delta}\right)
\]
with
\begin{small}
\[
\varphi_i=\left(\frac{z_iA_i(z_{i+1},z_{i+2},\ldots,z_{n-1})+B_i(z_{i+1},z_{i+2},\ldots,z_{n-1})}{z_iC_i(z_{i+1},z_{i+2},\ldots,z_{n-1})+D_i(z_{i+1},z_{i+2},\ldots,z_{n-1})}\right)
\in\mathrm{PGL}(2,\mathbb{C}(z_{i+1},z_{i+2},\ldots,z_{n-1}))
\]
\end{small}
and
\[
\frac{\alpha z_{n-1}+\beta}{\gamma z_{n-1}+\delta} \in\mathrm{PGL}(2,\mathbb{C}).
\]

According to the proof of Lemma \ref{lem:J0} one gets:

\begin{proposition}
Let $\varphi=(\varphi_0,\varphi_1,\ldots,\varphi_{n-1})$ be an 
element of $\mathrm{J}_n$.

Assume $0\leq i\leq n-2$. If $\det \varphi_i=\pm 1$, then 
$\mathfrak{n}\big(\varphi_i,\mathrm{PGL}(2,\mathbb{C}(z_{i+1},z_{i+2},\ldots,z_{n-1}))\big)\leq~4$
otherwise 
$\mathfrak{n}\big(\varphi_i,\mathrm{PGL}(2,\mathbb{C}(z_{i+1},z_{i+2},\ldots,z_{n-1}))\big)\leq 8$.

In particular $\mathfrak{n}\big(\varphi,\mathrm{J}_n\big)\leq 4(2n-1)$.
\end{proposition}

\begin{corollary}
Any element of 
$\mathrm{N}\big(\mathrm{J}_n;\mathrm{Bir}(\mathbb{P}^n_\mathbb{C})\big)$ 
can be written as a composition of involutions of $\mathbb{P}^n_\mathbb{C}$.
\end{corollary}

\bibliographystyle{plain}
\bibliography{biblio}

\end{document}